\documentclass[11pt]{amsart}
\usepackage{amsmath,amsfonts,amsthm,amssymb,
amsthm,graphicx,mathtools,amscd}
\usepackage{hyperref}

\usepackage[mathscr]{eucal}
\usepackage{graphics}
\usepackage{color}
\usepackage{epsfig}

\usepackage[nocompress]{cite}

\DeclareSymbolFont{iwonaletters}{OML}{iwona}{m}{it}
\DeclareMathSymbol{\bdel}{\mathalpha}{iwonaletters}{"E}

\DeclareMathOperator*{\argmin}{arg\,min}

\theoremstyle{plain}
\newtheorem{theorem}{Theorem}[section]
\newtheorem{lemma}[theorem]{Lemma}

\newtheorem{assumption}[theorem]{Assumption}

\newtheorem{remark}[theorem]{Remark}

\newcommand{\beginsec}{
\setcounter{equation}{0}
}

\newcommand{\la}{\lambda}
\newcommand{\eps}{\varepsilon}

\newcommand{\al}{\alpha}

\newcommand{\gam}{\gamma}
\newcommand{\kap}{\kappa}
\newcommand{\sig}{\sigma}
\newcommand{\del}{\delta}
\newcommand{\om}{\omega}

\newcommand{\Del}{\mathnormal{\Delta}}

\newcommand{\Ps}{\mathnormal{\Psi}}
\newcommand{\Om}{\mathnormal{\Omega}}

\newcommand{\N}{{\mathbb N}}

\newcommand{\R}{{\mathbb R}}

\newcommand{\Z}{{\mathbb Z}}

\newcommand{\EE}{{\mathbb E}}
\newcommand{\E}{{\mathbb E}}

\newcommand{\PP}{{\mathbb P}}

\newcommand{\calB}{{\mathcal B}}

\newcommand{\calF}{{\mathcal F}}

\newcommand{\calM}{{\mathcal M}}

\newcommand{\calQ}{{\mathcal Q}}
\newcommand{\calP}{{\mathcal P}}
\newcommand{\calR}{{\mathcal R}}
\newcommand{\calS}{{\mathcal S}}
\newcommand{\calT}{{\mathcal T}}

\newcommand{\calX}{{\mathcal X}}

\newcommand{\calZ}{{\mathcal Z}}

\newcommand{\oo}{\overline}
\newcommand{\supp}{{\rm supp}}

\newcommand{\w}{\wedge}

\newcommand{\pl}{\partial}

\newcommand{\To}{\Rightarrow}

\newcommand{\iy}{\infty}
\newcommand{\up}{\uparrow}

\newcommand{\cadlag}{c\`adl\`ag }

\newcommand{\noi}{\noindent}
\newcommand{\ds}{\displaystyle}

\newcommand{\leb}{\text{\rm Leb}}

\newcommand{\F}{F}

\begin{document}

\title{A Durrett-Remenik particle system in $\R^d$
}

\author{Rami Atar}
\address{\hspace{-1.3em}Viterbi Faculty of Electrical and Computer Engineering,
Technion -- Israel Isnstitute of Technology
{\rm August 19, 2025}
} 

\begin{abstract}
This paper studies a branching-selection model of motionless particles
in $\R^d$, with nonlocal branching,
introduced by Durrett and Remenik in dimension 1.
The assumptions on the fitness function, $F$, and on
the inhomogeneous branching distribution, are mild.
The evolution equation for the macroscopic density is given by an
integro-differential free boundary problem in $\R^d$,
in which the free boundary represents the least $F$-value in the population.
The main result is the characterization of
the limit in probability of the empirical measure process
in terms of the unique solution to this free boundary problem.
{\let\thefootnote\relax\footnote{{
\!\!{\it Key words.}
branching-selection particle systems,
hydrodynamic limits, free boundary problems

{\it Mathematics Subject Classification.}
35R35, 60J80, 82C22
}}}
\end{abstract}

\maketitle

\section{Introduction}

A system of particles living in $\R^d$,
$d\ge1$, that undergo branching and selection, is described as follows.
An initial configuration is given by $\{x^i,1\le i\le N\}$, where $x^i$ are $\R^d$-valued random variables. An independent Poisson clock of rate 1 is attached to each particle, according to which it gives birth to new ones.
The location of a particle born from one located at $y$ is drawn according
to a probability density $\rho(y,\cdot)$ defined on $\R^d$. All particles stay put throughout their lifetime. When a particle is born, the particle that has
the least $\F$-value among the living ones (including the newborn)
is removed, where $\F:\R^d\to\R$ is a given `fitness' function.
Tie breaking (which may be needed for the initial particles) occurs according to particle labels.
Note that the number of particles remains $N$ at all times.
The label of the particle that gets removed is transferred
to the new particle, so that at all times, the living particles are labeled
as $1,2,\ldots,N$. The location of particle $i$ at time $t$
is denoted by $X_i(t)$ (suppressing the dependence on $N$),
the configuration measure by $\xi^N_t=\sum\bdel_{X_i(t)}$,
and the empirical measure by $\bar\xi^N_t=N^{-1}\xi^N_t$. This model was introduced and studied in \cite{dur11} for $d=1$ and $\F(x)=x$,
where its hydrodynamic limit was characterized in terms of
an integro-differential free boundary problem (FBP).

This result is extended in
this paper. Under mild assumptions on $\rho$ and $\F$, it is shown
that the empirical measure process converges in probability
to a deterministic path in measure space, whose density uniquely solves
an integro-differential FBP in $\R^d$. This is
the problem of finding a pair $(u,\ell)$, where $\ell:\R_+\to\R$
is \cadlag and the function $u:\R^d\times\R_+\to\R_+$ is $C^{0,1}$ in
$\{(x,t)\in\R\times\R_+:\F(x)>\ell_t\}$, and one has
\[
\begin{cases}\ds
\pl_tu(x,t)=\int_{\R^d}u(y,t)\rho(y,x)dy & \quad (x,t):\F(x)>\ell_t,\,t>0,\\
u(x,t)=0 & \quad (x,t):\F(x)\le\ell_t,\,t>0, \\
\ds
\int_{\R^d}u(x,t)dx=1, & \quad t>0, \\
u(\cdot,0)=u_0,\,\ell_0=\la_0.
\end{cases}
\]
The function $u$ and the free boundary $\ell$ represent the density and, respectively, the least $F$-value in the population at the macroscopic scale, with $u_0$ and $\la_0$ their initial conditions.

The relation between particle system with selection and FBP is an active research field. Particle systems with selection were proposed in \cite{bru06, bru07} as models for natural selection in population dynamics, where the position of a particle on the real line represents the degree of fitness of an individual to its environment, and those that are least fitted get removed. In \cite{dur11}, the aforementioned Durrett-Remenik model for motionless particles was introduced, where, in addition to hydrodynamic limits, traveling wave solutions were also studied.
The monograph \cite{de-masi-book} studied hydrodynamic limits of Brownian particle systems with injection and selection, relating them to FBP. A model for Brownian particles with branching and selection, referred to as $N$-particle branching Brownian motion ($N$-BBM), related to the ones proposed in \cite{bru06, bru07}, was introduced in \cite{mai16}, and studied at the hydrodynamic limit in \cite{de-masi-nbbm, ber19}. A variant of $N$-BBM with nonlocal branching and its relation to FBP was studied in \cite{de-masi-nbbm2}. Other hydrodynamic limit results for branching-selection models were studied in \cite{groisman2021rank, ata25}. Apart from selection models, particle systems that give rise to FBP include a simple exclusion process with boundary behavior \cite{de-masi-excl} and the Atlas model \cite{dem19, ata-bud-25}.

The works mentioned above all address FBP in dimension 1.
In higher dimension, \cite{n.ber18} studied the asymptotic shape
of the cloud of particles of the $N$-BBM with fitness function $\F(x)=\|x\|$ and
$\F(x)=\la\cdot x$. The papers \cite{ber21, ber22bee} considered
the hydrodynamic limit of the $N$-BBM with $\F(x)=-\|x\|$,
characterizing it in terms of a FBP, studying convergence rates
as well as long time behavior.
This treatment used in a crucial way the symmetry of $\F$,
by which the radial projection of the macroscopic dynamics
is governed by an autonomous equation, and thus the motion of the free boundary
is dictated by an equation in one dimension.

Our motivation in this work is to study a branching-selection model for which the macroscopic evolution equation is truly multidimensional, in the sense that it does not reduce to a FBP in one spatial dimension.
This is relevant also from the application viewpoint. For example, in the context of natural selection, consider scores given to different traits, each represented by one of the coordinates $x_i$ of the state $x$ of an individual. In this case, the fitness functions $F(x)=\min_i x_i$ and $F(x)=\max_ix_i$ may be of interest, as they represent scenarios where the worst (resp., best) score among the various traits governs extinction.

The Durrett-Remenik and $N$-BBM models in $\R^d$ are both multidimensional branching-selection models, and so they share some similarities, but there are significant differences as far as their analysis is concerned. The most obvious one is an integro-differential FBP studied here for the former versus a second order parabolic FBP studied in \cite{ber21, ber22bee} for the $N$-BBM.
The tools required are different, where, for example, the very question of existence of classical solutions to the PDE of \cite{ber21, ber22bee} is difficult. Their results not only obtain the hydrodynamic limit but also large time asymptotics, which we do not address here.
In this paper, FBP uniqueness builds on monotonicity of the free boundary $\ell$, whereas this property does not hold in general for $N$-BBM. Finally, for the $N$-BBM, establishing the hydrodynamic limit for general fitness functions under which the dynamics do not reduce to one dimension is a challenging open question.

The rest of this article is organized as follows.
In \S \ref{sec2}, assumptions are introduced, the main result is stated, and the proof is outlined. The proof is given in \S\ref{sec3},
starting with \S\ref{sec31} where uniqueness of solutions to the FBP is proved.
In \S\ref{sec32}, it is shown that the empirical measure processes
form a tight sequence and that limits are supported on solutions
to the FBP, which completes the proof. A construction required for one of the steps in \S\ref{sec32} is provided in the appendix.


{\it Notation.}
Denote $\R_+=[0,\iy)$. Let $\calB(\R^d)$ denote the class of Borel subsets of $\R^d$.
In $\R^d$, denote the Euclidean norm by $\|\cdot\|$
and let $\mathbb{B}_r=\{x\in\R^d:\|x\|\le r\}$.
Denote by $\calM(\R^d)$ the space of finite signed Borel measures on $\R^d$
endowed with the topology of weak convergence.
Let $\calP(\R^d)\subset\calM_+(\R^d)\subset\calM(\R^d)$
denote the subsets of probability and, respectively positive measures,
and give them the inherited topologies.
For $\mu,\nu\in\calM_+(\R^d)$,
write $\mu\sqsubset\nu$ if $\mu(A)\le\nu(A)$ for all $A\in\calB(\R^d)$.

For $(X,d_X)$ a Polish space let $C(\R_+,X)$ and $D(\R_+,X)$
denote the space of continuous and, respectively, \cadlag paths,
endowed with the topology of uniform convergence on compacts
and, respectively, the Skorohod $J_1$ topology.
Let $C^\up(\R_+,\R_+)$ denote the subset of $C(\R_+,\R_+)$
of nondecreasing functions that vanish at zero.
For $f:\R_+\to\R^d$ denote
\begin{align*}
w_T(f,\del)&=\sup\{\|f(t)-f(s)\|:0\le s\le t\le (s+\del)\w T\},
\\
\|f\|^*_T&=\sup\{\|f(s)\|:s\in[0,T]\}.
\end{align*}

\section{Assumptions and main result}\label{sec2}
\beginsec

\subsection{Assumptions and marcoscopic dynamics}
Recall that the empirical measure is given by $\xi^N_t=N^{-1}\sum\bdel_{X_i(t)}$.
In particular, the initial empirical measure is $\bar\xi^N_0=N^{-1}\sum\bdel_{x^i}$.
Following is our assumption on $\F$, $\rho$ and $\bar\xi^N_0$.
Denote $\ell^N_0=\min\{F(y):y\in\supp(\xi^N_0)\}$.

\begin{assumption}\label{assn10} (i)
$\F\in C(\R^d,\R)$, $\inf_x\F(x)=-\iy$, $\sup_x\F(x)=\infty$, and
for every $a\in\R$, $\F^{-1}\{a\}$ has Lebesgue measure zero.

(ii)
There exists a probability density $\tilde\rho$ and a constant $\tilde c$
such that $\rho(x,y)\le \tilde c\tilde\rho(y-x)$.
Moreover, $\rho(x,y)$ is continuous in $y$ uniformly in $(x,y)$, and
for $a\in\R$ and $x\in \F^{-1}(a,\iy)$, one has
$\int_{\F^{-1}(a,\iy)}\rho(x,y)dy>0$.

(iii)
As $N\to\iy$, $(\bar\xi^N_0,\ell^N_0)\to(\xi_0,\la_0)$ in probability, 
where the latter tuple is deterministic,
$\xi_0(dx)=u_0(x)dx$ and $\la_0\in\R$.
Moreover, $u_0$ is bounded and continuous on $\F^{-1}(\la_0,\iy)$
(and necessarily vanishes on $F^{-1}(-\iy,\la_0)$), and
for every $\del>0$ and $\la\ge \la_0$, $\int_{\F^{-1}(\la,\la+\del)} u_0(x)dx>0$.
\end{assumption}

Further notation is $\ell^N_t=\min\{F(y):y\in\supp(\xi^N_t)\}$
for the minimal $\F$-value of all living particles at time $t$,
and $J^N$ for the removal counting process. By construction,
$J^N$ is a Poisson process of rate $N$.

Next is an equation for the macroscopic dynamics, in the form of an integro-differential FBP.
Denote by $\calX$ the set of pairs $(u,\ell)$, where $\ell\in D(\R_+,\R)$
and $u:\R^d\times\R_+\to\R_+$ is $C^{0,1}$ in
$\{(x,t)\in\R\times\R_+:\F(x)>\ell_t\}$ and bounded on
$\R^d\times[0,T]$ for any $T$.
Consider the system
\begin{equation}\label{a1}
\begin{cases}\ds
{\it (i)}\hspace{3.5em}
\pl_tu(x,t)=\int_{\R^d}u(y,t)\rho(y,x)dy & \quad (x,t):\F(x)>\ell_t,\,t>0,\\
{\it (ii)}\hspace{3.1em}
u(x,t)=0 & \quad (x,t):\F(x)\le\ell_t,\,t>0, \\
{\it (iii)}\hspace{2.5em}
\ds
\int_{\R^d}u(x,t)dx=1, & \quad t>0, \\
{\it (iv)}\hspace{2.9em}
u(\cdot,0)=u_0,\,\ell_0=\la_0.
\end{cases}
\end{equation}
A solution to \eqref{a1} is defined as a member of $\calX$ satisfying \eqref{a1}.

\subsection{Main result}

\begin{theorem}\label{th3}
Let Assumption \ref{assn10} hold. Then there exists a unique solution
to \eqref{a1}, denoted by $(u,\ell)$. Moreover,
$\ell-\la_0\in C^\up(\R_+,\R_+)$.
Furthermore, $(\bar\xi^N,\ell^N)\to(\xi,\ell)$, in probability,
in $D(\R_+,\calP(\R^d))\times D(\R_+,\R_+)$,
where $\xi_t(dx)=u(x,t )dx$.
\end{theorem}

\begin{remark}
Most parts of Assumption \ref{assn10} are quite mild. The parts that may require some comment are
the last part of Assumption \ref{assn10}(ii) and the last part of Assumption \ref{assn10}(iii). These are used to show that $\ell$ is monotone whenever $(u,\ell)$ is a solution of the FBP (see Lemma \ref{lem51}), and, respectively, that $\ell^N$ are $C$-tight (see Lemma \ref{lem53}).
\end{remark}

\begin{remark}
Because, as stated above, the $\ell$ component
of any solution to \eqref{a1} is nondecreasing,
\eqref{a1}(i) can be written in integral form as
\[
u(x,t)=u_0(x)+\int_0^t\int_{\R^d}u(y,s)\rho(y,x)dydt,\qquad (x,t):F(x)>\ell_t,t>0.
\]
In a more general setting not covered in this paper,
where the mass conservation condition is replaced by
$\int_{\R^d}u(x,t)dx=m(t)$, with $m(t)$ given,
$\ell$ need not be nondecreasing. In this case
the system \eqref{a1} (with $1$ replaced by $m(t)$ on the r.h.s.\ of (iii))
is not sufficient for characterizing $(u,\ell)$. Roughly speaking,
a boundary condition $u(\ell_t,t)=0$ should be added at times
when $\ell_t$ is decreasing. A precise way to write this is
\[
u(x,t)=u_0(x)1_{\{\tau(t,x)=0\}}
+\int_{\tau(t,x)}^t\int_{\R^d}u(y,s)\rho(y,x)dydt,\qquad (x,t):F(x)>\ell_t,t>0,
\]
where $\tau(x,t)=\inf\{s\in[0,t]:\ell_\theta<x \text{ for all } \theta\in(s,t)\}$.
\end{remark}

\begin{remark}\label{rem:gap}
Our proof corrects a gap in the proof of uniqueness of solutions to the integro-differential FBP in \cite{dur11}
(this is equation (FB),
corresponding to \eqref{a1} above in the case $d=1$ and $F(x)=x$;
see \cite[Remark 2.10]{dur11-arxiv} for details).
Addressing uniqueness via a different approach, our treatment validates the uniqueness statement of \cite{dur11}.
\end{remark}

\subsection{Proof outline}\label{sec23}

The proof of Theorem \ref{th3} is based on the compactness--uniqueness approach, and proceeds in three main steps.

(a) {\it Uniqueness.} This is the content of \S\ref{sec31}, where it is shown that \eqref{a1} has at most one solution.

For a quick sketch of the idea, consider the case
$d=1$ and $F(x)=x$. If $(u,\ell)$ and $(v,m)$
are solutions then, for each $t>0$, $u(x,t)$ vanishes for $x<\ell_t$ and
$v(x,t)$ vanishes for $x<m_t$. This and the fact that $u$ and $v$ have the same mass
imply
\[
\int_\R|u(x,t)-v(x,t)|dx\le 2\int_{\ell_t\vee m_t}^\iy|u(x,t)-v(x,t)|dx.
\]
For both $u$ and $v$, the r.h.s.\ now involves only $(x,t)$
for which the integro-differential equation \eqref{a1}(i) holds.
Integrating it over time allows the use of Gronwall's lemma.

(b) {\it Tightness.}
In particular, in Lemma \ref{lem53} of \S\ref{sec32}, it is shown that $(\bar\xi^N,\ell^N)$ forms a $C$-tight sequence.

(c) {\it Subsequential limits form FBP solutions.} This is shown in the remainder of \S\ref{sec32}.

The idea here is to use the standard fact that normalized configuration measures of particle systems with fixed boundary converge as $N\to\iy$ to solutions to integro-differential equations with fixed boundary.

To exploit this fact, upper and lower piecewise constant (but random) envelopes of the prelimit free boundary $\ell^N$ are selected. Then, auxiliary particle systems, with piecewise constant boundaries given by these envelopes, are constructed. It is argued that the configuration measures of the auxiliary particle systems form lower and upper bounds on $\xi^N_t$, in the sense of measure inequalities. The aforementioned convergence as $N\to\iy$ can then be applied in order to sandwich any subsequential limit of $\bar\xi^N_t$ in terms of solutions to the corresponding integro-differential equations. Because $\ell^N$ are $C$-tight, the upper and lower envelopes can be made arbitrarily close to each other locally uniformly. A lemma showing that the solutions of the auxiliary integro-differential equations are continuous with respect to perturbations in the boundary then completes the argument.

\section{Proof of main result}\label{sec3}
\beginsec

\subsection{Uniqueness}\label{sec31}

Before implementing the argument outlined in \S\ref{sec23} part (a), we need the following.

\begin{lemma}\label{lem51}
Let $(u,\ell)\in\calX$ be a solution of \eqref{a1}.
Then $\ell$ is nondecreasing.
\end{lemma}

\begin{proof}
Arguing by contradiction, assume
there exists $t>0$ such that $\ell_t<L_t:=\sup_{s\in[0,t]}\ell_s$.
There are two possibilities.

1. There is $s<t$ such that $\ell_s=L_t$.
In this case, $\ell_\theta\le\ell_s$ for all $\theta\in[s,t]$.
If for some $x$ $\F(x)>\ell_s$ then $\F(x)>\ell_\theta$ for all
$\theta\in[s,t]$, and we can integrate \eqref{a1}(i). Thus
\begin{equation}\label{63}
u(x,t)=u(x,s)+\int_s^t\int_{\R^d}u(y,\theta)\rho(y,x)dyd\theta,
\qquad x\in \F^{-1}(\ell_s,\iy).
\end{equation}
We obtain
\begin{align*}
1=\int_{\R^d} u(x,t)dx&\ge\int_{\F^{-1}(\ell_s,\iy)} \Big[u(x,s)
+\int_s^t\int_{\R^d} u(y,\theta)\rho(y,x)dy d\theta \Big]dx\\
&\ge
1 +\int_s^t\int_{\F^{-1}(\ell_s,\iy)}\int_{\F^{-1}(\ell_s,\iy)} u(y,\theta)\rho(y,x)dy dx d\theta.
\end{align*}
For every $y\in \F^{-1}(\ell_s,\iy)$, one has
$\psi(y):=\int_{\F^{-1}(\ell_s,\iy)}\rho(y,x)dx>0$, by Assumption \ref{assn10}(ii).
Since $\ell_s\ge\ell_\theta$ for all $\theta\in[0,t]$,
we have, similarly to \eqref{63}, that
\[
u(y,\theta)=u_0(y)+\int_0^\theta\int_{\R^d}u(y',\theta')\rho(y',y)dy'd\theta',
\qquad y\in \F^{-1}(\ell_s,\iy).
\]
Hence for such $y$ and all $\theta\in[s,t]$ one has $u(y,\theta)\ge u_0(y)$.
Hence the triple integral above is bounded below by
\[
(t-s)\int_{\F^{-1}(\ell_s,\iy)} u_0(y)\psi(y)dy.
\]
But $\int_{\F^{-1}(\ell_s,\iy)}u_0(y)dy>0$ by assumption, hence
the above integral is positive, a contradiction.

2. There is $s\le t$ such that $\ell_{s-}=L_t$, and $\ell_s<\ell_{s-}$.
In this case there exists $t'>s$ such that $\ell_\theta\le\ell_{s-}$ for all $\theta\in[s,t']$.
We first show
\begin{equation}\label{62}
\int_{\F^{-1}(\ell_{s-},\iy)} u(x,s)dx=1.
\end{equation}
Let $s_n\uparrow s$.
Then $\ell_{s_n}\to\ell_{s-}$ and by assumption,
$\ell_{s_n}\le\ell_{s-}$. Now,
\begin{align}\label{64}
\int_{\F^{-1}(\ell_{s-},\iy)}u(x,s)dx
&=\int_{\F^{-1}(\ell_{s_n},\iy)}u(x,s_n)dx
+\int_{\F^{-1}(\ell_{s_n},\iy)}(u(x,s)-u(x,s_n))dx
\\
&\quad-\int_{\F^{-1}(\ell_{s_n},\ell_{s-})} u(x,s)dx.
\notag
\end{align}
The first term on the right is $1$, and the last term converges to zero
as $n\to\iy$. As for the second term,
since for all $\theta\in[s_n,s]$ one has
$\F^{-1}(\ell_{s-},\iy)\subset \F^{-1}(\ell_{\theta},\iy)$,
one can integrate in \eqref{a1}(i) and get
\[
\int_{\F^{-1}(\ell_{s-},\iy)}|u(x,s)-u(x,s_n)|dx
\le\int_{s_n}^s\int_{\R^d}u(y,\theta)dyd\theta=s-s_n.
\]
The above expression and the second term in \eqref{64}
have the same limit by the assumed boundedness of $u$
on $\R^d\times[0,s]$. This shows \eqref{62}.

Using \eqref{62} and $\F^{-1}(\ell_{s-},\iy)\subset \F^{-1}(\ell_{t'},\iy)$,
we have $\int_{\F^{-1}(\ell_{t'},\iy)} u(x,s)dx=1$.
We can thus repeat the argument above in 1, with
$(s,\ell_{s-},t',\ell_{t'})$ in place of $(s,\ell_s,t,\ell_t)$.
Namely,
\begin{align*}
1=\int_{\F^{-1}(\ell_{t'},\iy)} u(x,t')dx&=\int_{\F^{-1}(\ell_{t'},\iy)} \Big[u(x,s)
+\int_s^{t'}\int_{\R^d} u(y,\theta)\rho(y,x)dy d\theta \Big]dx\\
&\ge
1 +\int_s^{t'}\int_{\F^{-1}(\ell_{s-},\iy)}^\iy\int_{\F^{-1}(\ell_{s-},\iy)} u(y,\theta)\rho(y,x)dy dx d\theta.
\end{align*}
The argument now completes exactly as in case 1.
\end{proof}

\begin{lemma}\label{lem52}
Let $(u,\ell)$ and $(v,m)$ be solutions of \eqref{a1}.
Then $(u,\ell)=(v,m)$.
\end{lemma}

\begin{proof}
If $w\in L^1(\R^d)$ and $\int_{\R^d}w(x)dx=0$ then we have
$\|w\|_1=2\|w^+\|_1=2\|w^-\|_1$, where we denote $\|\cdot\|_1=\|\cdot\|_{L^1}$.
If in addition $w\ge 0$ on some domain $D$ then
\[
2\|w^-\|_1=2\|w^-1_{D^c}\|_1\le 2\|w1_{D^c}\|_1.
\]
A similar statement holds with $w\le0$ and $w^+$.
Hence if $w$ is either nonnegative on $D$ or nonpositive on $D$,
\begin{equation}\label{001}
\|w\|_1\le2\|w1_{D^c}\|_1.
\end{equation}

Consider $(x,t)$ such that $\F(x)>\ell_t$. Then $\F(x)>\ell_s$ for all $s\le t$.
Therefore \eqref{a1}(i) is valid with $(x,t)$ replaced by $(x,s)$ for all such $s$, and
\[
u(x,t)=u_0(x)+\int_0^t\int_{\R^d}u(y,s)\rho(y,x)dyds.
\]
Similarly, if $\F(x)>m_t$, the above is satisfied by $v$.
Denote $\Del=u-v$. Consider $(x,t)$ such that $\F(x)>\ell_t\vee m_t$. Then
\begin{equation}\label{002}
\Del(x,t)=\int_0^t\int_{\R^d}\Del(y,s)\rho(y,x)dyds.
\end{equation}
Now, for each $t$, $\int_{\R^d}\Del(x,t)dx=1-1=0$. Moreover, in
$\F^{-1}(-\iy,\ell_t\vee m_t)$, either $u$ or $v$ vanishes,
therefore $\Del(\cdot,t)$ is either nonnegative or nonpositive.
Hence we can apply \eqref{001} and then \eqref{002} to get
\begin{align*}
\int_{\R^d}|\Del(x,t)|dx
&\le2\int_{\F^{-1}(\ell_t\vee m_t,\iy)}|\Del(x,t)|dx
\\
&\le 2\int_{\F^{-1}(\ell_t\vee m_t,\iy)}\int_0^t\int_{\R^d}|\Del(y,s)|\rho(y,x)dydsdx
\\
&\le 2\int_0^t\int_{\R^d}|\Del(y,s)|dyds.
\end{align*}
The above is true for all $t\ge0$, hence by Gronwall's lemma,
the integral on the left vanishes for all $t$.
This shows that for every $t$, $u(x,t)=v(x,t)$ for
a.e.\ $x$. Next, for every $t$, both $u(\cdot,t)$ and $v(\cdot,t)$
are continuous in each of the domains
$\{F(x)\le\ell_t\w m_t\}$, $\{\ell_t\w m_t<F(x)\le\ell_t\vee m_t\}$
and $\{F(x)>\ell_t\vee m_t\}$,
and therefore must be equal everywhere. This shows $u=v$.

It remains to show that $\ell=m$. Since $u=v$,
$(\ell,u)$ and $(m,u)$ are solutions. Arguing by contradition,
assume that, say, $\ell_\theta<m_\theta$ for some $\theta$.
By right continuity, there exists an interval $[\theta,\theta_1]$ such that
\[
\sup_{t\in[\theta,\theta_1]}\ell_t < \inf_{t\in[\theta,\theta_1]}m_t.
\]
Let $\hat\ell$ be defined by $\hat\ell_t=m_t$ for $t<\theta_1$ and
$\hat\ell_t=\ell_t$ for $t\ge \theta_1$. Then $\hat\ell$ is \cadlag.
Moreover, the differential equation \eqref{a1}(i) holds when
$\F(x)>\hat\ell_t$ (because it holds in the larger domain $\F(x)>\ell_t$),
and the vanishing condition \eqref{a1}(ii) holds when $\F(x)\le\hat\ell_t$
(because it holds in the larger domain $\F(x)\le m_t$).
Hence $(\hat\ell,u)$ is a solution.
However, by construction, $\hat\ell$ is not a nondecreasing trajectory,
which contradicts the monotonicity property proved earlier.
This shows that $\ell=m$.
\end{proof}

\subsection{Proof of Theorem \ref{th3}}\label{sec32}

In this section it is proved first, in Lemma \ref{lem53}, that tightness holds,
and then the main remaining task is to show that limits satisfy \eqref{a1}(i). As outlines in \S\ref{sec32}, this is achieved by constructing upper and lower bounds on the density
given in terms of limits of particle systems with piecewise constant boundary,
for which convergence is a consequence of earlier work. These piecewise constant
boundaries are constructed to form upper and lower envelopes of the prelimit
free boundary $\ell^N$.

\begin{lemma}\label{lem53}
The sequence of laws of $(\bar\xi^N,\ell^N)$ is $C$-tight.
Moreover, every subsequential limit $(\xi,\ell)$ satisfies a.s.,
\[
\xi_t(F^{-1}(-\iy,\ell_t))=0, \text{ for all } t\in(0,\iy).
\]
\end{lemma}

\begin{proof}
We first argue $C$-tightness of $\ell^N$.
Since $J^N$ is a rate-$N$ Poisson process, $\bar J^N=N^{-1}J^N$ converges in probability to the identity map from $\R_+$ to itself.
Next, by construction, the path $t\mapsto\ell^N_t$
is nondecreasing, because when a new particle
is born at $t$ in the domain $\F^{-1}[\ell^N_{t-},\iy)$
one has $\ell^N_t\ge\ell^N_{t-}$, and if it is born outside this domain,
it is removed immediately and $\ell^N_t=\ell^N_{t-}$.
In particular, $\ell^N_t\ge \ell^N_0\to \la_0$ in probability.
Fix $T>0$. We show that the random variables $\ell^N_T$ are tight.
To this end, denote by $\zeta^N$ the configuration measure
associated with non-local branching without removals. That is, they are
constructed as our original system, but without removing any particles,
and the systems are coupled so that at all times, the configuration
of the original system is a subset of that of the enlarged one.
Then there exists a (deterministic) finite measure on $\R^d$,
$\zeta_T$, such that $\bar\zeta^N_T\to\zeta_T$ in probability
by \cite[Theorem 5.3]{fou-mel}. As a result,
there exists a compact $K\subset\R^d$ such that
$\limsup_N\PP(\bar\zeta^N_T(K^c)>\frac{1}{2})=0$.
Let $k=\max\{\F(x):x\in K\}$. Then
$\limsup_N\PP(\bar\zeta^N_T(\F^{-1}(k,\iy))>\frac{1}{2})=0$.
Since $\xi^N_T$ is dominated by $\zeta^N_T$ a.s.,
\[
\PP(\ell^N_T>k)=\PP(\xi^N_T(\F^{-1}(k,\iy))=N)
\le \PP(\bar\zeta^N_T(\F^{-1}(k,\iy))\ge1),
\]
showing $\limsup_N\PP(\ell^N_T>k)=0$.

Next, for $\ell^N$ to increase over a time interval $[t,t+h]$
by more than $\del$, the particles located at time $t$ in
$D^N(t,\del):=\F^{-1}[\ell^N_t,\ell^N_t+\del)$ must be removed by time $t+h$.
Because $\ell^N$ is monotone, all particles
in the initial configuration within the domain $D^N(t,\del)$
are still present at time $t$.
Hence the event $\ell^N_{t+h}>\ell^N_t+\del$
is contained in
\[
\xi^N_0(D^N(t,\del))\le J^N_{t+h}-J^N_t.
\]
As a result, the event $w_T(\ell^N,h)>\del$ is contained in
\[
\inf_{t\in[0,T-h]}\bar\xi^N_0(D^N(t,\del))\le w_T(\bar J^N_t,h).
\]
Hence for any $\del,h,\eps>0$,
\begin{align*}
\PP(w_T(\ell^N,h)>\del)
&\le\PP(\ell^N_T>k)+\PP(w_T(\bar J^N,h)>\eps))
\\
&\quad
+\PP\Big(\inf_{t\in[0,T-h]}\bar\xi^N_0(D^N(t,\del))\le\eps,\,\ell^N_T\le k\Big).
\end{align*}
Thus
\begin{align}\label{202}
\notag
\limsup_N\PP(w_T(\ell^N,h)>\del)
&\le
\limsup_N\PP(w_T(\bar J^N,h)>\eps))
\\
&\quad
+\limsup_N\PP\Big(\inf_{a\in[\la_0,k]}\bar\xi^N_0(\F^{-1}[a,a+\del))<\eps\Big).
\end{align}
Let $\del>0$ be given.
By assumption, $\eta(a,\del):=\int_{\F^{-1}(a,a+\del)} u_0(x)dx>0$,
$a\ge\la_0$.
Moreover, by the boundedness of $u_0$, $a\mapsto\eta(a,\del)$
is continuous. Therefore $\inf_{a\in[\la_0,k]}\eta(a,\del)>0$.
This and the convergence $\bar\xi^N_0\to\xi_0=u_0(x)dx$
in probability show that
one can find $\eps>0$ such that the second term on
the r.h.s.\ of \eqref{202} vanishes.
Given such $\eps$, using the fact that limits of $\bar J^N$ are $1$-Lipschitz,
the first term on the r.h.s.\ of \eqref{202} also vanishes provided
$h$ is sufficiently small. This completes the proof of $C$-tightness
of $\ell^N$.

Next, $C$-tightness of $\bar\xi^N$ is shown.
Denote by $d_{\rm L}$ the Levy-Prohorov metric on $\calP(\R^d)$,
which is compatible with weak convergence on this space. We will show that
(i) for every $\eps>0$ and $t$ there exists a compact set $K_{\eps,t}\subset\calP(\R^d)$
such that $\liminf_N\PP(\bar\xi^N_t\in K_{\eps,t})>1-\eps$;
and (ii)
for every $\eps>0$ there exists $\del>0$ such that
\[
\limsup_N \PP(w_T(\bar\xi^N,\del)>\eps)<\eps.
\]
This will establish $C$-tightness of $\bar\xi^N$ in view of
\cite[Corollary 3.7.4 (p.\ 129)]{ethkur} and because we use $w$ rather than $w'$
\cite[(3.6.2) (p.\ 122)]{ethkur}.

To show (i), let
\[
K_n(r)=\{\gam\in\calP(\R^d):\gam(\mathbb{B}_r^c)<n^{-1}\},
\qquad n\in\N,\,r\in(0,\iy).
\]
By Prohorov's theorem, for any $n_0$ and sequence $\{r_n\}$,
the closure of $K_{\ge n_0}(\{r_n\}):=\cap_{n\ge n_0}K_n(r_n)$,
as a subset of $\calP(\R^d)$, is compact. Suppose we show that
there exists a sequence $r_n$ such that, for every $t\in[0,T]$,
\begin{equation}
\label{203}
\liminf_N\PP(\bar\xi^N_t\in K_n(r_n))\ge1-2^{-n}.
\end{equation}
Then, given $\eps>0$, taking $n_0$ such that
$\sum_{n\ge n_0}2^{-n}<\eps$, it would follow that
\[
\liminf_N\PP(\bar\xi^N_t\in \oo{K_{\ge n_0}(\{r_n\})})>1-\eps,
\]
showing that (i) holds.
To this end, note that the convergence of $\bar\zeta^N_T$ implies that
there exists a sequence $\{r_n\}$ such that
\[
\liminf_N\PP(\bar\zeta^N_T(\mathbb{B}_{r_n}^c)<n^{-1})>1-2^{-n}.
\]
But
\[
\sup_{t\in[0,T]}\bar\xi^N_t(\mathbb{B}_{r_n}^c)
\le\bar\zeta^N_T(\mathbb{B}_{r_n}^c),
\]
hence \eqref{203} follows, and (i) is proved.

Next we show (ii).
Given a set $C\subset\R$ let $C^\eps$ denote its $\eps$-neighborhood.
Let $0\le s<t\le T$ be such that $t-s\le\del$. Then for any Borel set $C$, one has
\[
|\xi^N_t(C)-\xi^N_s(C)|\le J^N_t-J^N_s.
\]
Hence, on the event $w_T(\bar J^N,\del)\le\eps$,
\[
\bar\xi^N_s(C) \le \bar\xi^N_t(C^\eps) + \eps,
\qquad \text{and} \qquad
\bar\xi^N_t(C) \le \bar\xi^N_s(C^\eps)+ \eps,
\]
and thus $d_{\rm L}(\bar\xi^N_s,\bar\xi^N_t)\le\eps$.
This shows that for sufficiently small $\del$,
\[
\limsup_N \PP(w_T(\bar\xi^N,\del)>\eps)\le\eps,
\]
and the proof of (ii) is complete.

For the second assertion of the lemma, note that
by the definition of $\ell^N$, one has for all $N$,
$\xi^N_t(F^{-1}(-\iy,\ell^N_t))=0$.
Invoking Skorohod's representation, one has
$(\bar\xi^N,\ell^N)\to(\xi,\ell)$ a.s.\ along the convergent subsequence.
Hence for $t>0$ and $\eps>0$, because $F^{-1}(-\iy,\ell_t-\eps)$ is open,
\begin{align*}
\xi_t(F^{-1}(-\iy,\ell_t-\eps)) &\le \liminf_N\bar\xi^N_t(F^{-1}(-\iy,\ell_t-\eps))
\\
&\le\liminf_N\bar\xi^N_t(F^{-1}(-\iy,\ell^N))=0,
\end{align*}
a.s.
Taking $\eps\downarrow0$, $\xi_t(F^{-1}(-\iy,\ell_t))=0$, a.s.
To deduce the a.s.\ statement simultaneously for all $t$, it suffices to
note that for $t_n\downarrow t$, by monotonicity of $\ell$, one has
\[
\xi_t(F^{-1}(-\iy,\ell_t))\le\liminf_n\xi_{t_n}(F^{-1}(-\iy,\ell_t))\le\liminf_n\xi_{t_n}(F^{-1}(-\iy,\ell_{t_n}))=0.
\]
\end{proof}

\begin{proof}[Proof of Theorem \ref{th3}]
In view of Lemmas \ref{lem52} and \ref{lem53}, the proof
will be complete once it is shown that for every limit $(\xi,\ell)$
there exists a measurable density $u$ such $(u,\ell)\in\calX$
and $(u,\ell)$ satisfies \eqref{a1}.
Fix a convergent subsequence and denote its limit by $(\xi,\ell)$.

To argue the existence of a density let us go back to
the particle system with no removals,
mentioned in the proof of Lemma \ref{lem53}.
The normalized process $\bar\zeta^N\to\zeta$, in probability, where
$\zeta$ is deterministic and for every $t$, $\zeta_t$
has a density, as follows from
\cite[Theorem 5.3 and Proposition 5.4]{fou-mel}.
Throughout what follows, denote this density
by $z(\cdot,t)$.

For any bounded continuous $g:\R^d\to[0,\iy)$, one has
$\E\int g(x)\bar\xi^N_t(dx)\le\E\int g(x)\bar\zeta^N_T(dx)$, $t\in[0,T]$.
Hence $\E\int g(x)\xi_t(dx)\le\int g(x)\zeta_T(dx)$, $t\in[0,T]$.
In particular, $\xi_t(dx)\ll dx$, $t\le T$, and
since $T$ is arbitrary, this statement holds for all $t$.
Assuming $\xi=0$ outside
the full measure event, we finally obtain that
for every $(t,\om)\in(0,\iy)\times\Om$, $\xi_t(dx,\om)\ll dx$.
We can now appeal to \cite[Theorem 58 in Chapter V (p.\ 52)]{del-mey-82}
and the remark that follows.
The measurable spaces denoted in \cite{del-mey-82}
by $(\Om,\calF)$ and $(T,\calT)$ are taken to be
$(\R^d,\calB(\R^d))$ and $((0,\iy)\times\Om,\calB((0,\iy))\otimes\calF)$, respectively.
According to this result there exists
a $\calB(\R^d)\otimes\calB((0,\iy))\otimes\calF$-measurable function
$u(x,t,\om)$, such that for every $(t,\om)\in(0,\iy)\times\Om$,
$u(\cdot,t,\om)$ is a density of $\xi_t(dx,\om)$ with respect to $dx$.
We also have $u(x,t,\om)\le z(x,t)$.

Items (ii), (iii) and (iv) of \eqref{a1} can be verified plainly:
In view of Lemma \ref{lem53}, $u$ has a version satisfying
$u(x,t)=0$ for all $(x,t)$ such that $F(x)<\ell_t$, and
this extends to $F(x)\le\ell_t$ using Assumption \ref{assn10}(i) by which
$\leb\,F^{-1}\{\ell_t\}=0$ for all $t$.
This verifies \eqref{a1}(ii).
Items (iii) and (iv) are obvious from our assumptions.
It remains to show \eqref{a1}(i),
which by the nondecreasing property of $\ell$ can be expressed as
\begin{equation}\label{b2}
u(x,t)=u_0(x)+\int_0^t\int_{\R^d}u(y,s)\rho(y,x)dyds,
\qquad (x,t):F(x)>\ell_t,t>0.
\end{equation}
Note that as soon as the above is established,
the uniform continuity of $\rho$ and the local integrability of $u$
imply the existence of a version of $u$ for which
$x\mapsto u(x,t)$ and $t\mapsto\pl_tu(x,t)$ are continuous,
hence $(u,\ell)\in\calX$.
We will achieve \eqref{b2} in several steps.

Step 1.
{\it Integro-differential systems with piecewise constant boundary.}
Fix $T>0$ throughout.
For $\del,\eps\in(0,1)$ such that $H:=\del^{-1}T\in\N$,
let $M=M(\del,\eps)$
denote the collection of right-continuous
piecewise constant nondecreasing
trajectories $m:[0,T)\to\R$ such that, with $t_j=j\del$,
for each $j=0,1,\ldots,H-1$, on $[t_j,t_{j+1})$,
$m$ takes a constant value in $\eps\Z$.
Denote $a_j(m)=m(t_j)$.
For $m\in M$ and $a_j=a_j(m)$, consider the set of equations
\begin{equation}\label{b1}
\begin{cases}\ds
\pl_tu(x,t)=\int_{\R^d}u(y,t)\rho(y,x)dy &
(x,t):\F(x)>a_j,\,t\in(t_j,t_{j+1}),\\
u(x,t)=0 & (x,t):\F(x)\le a_j,\,t\in(t_j,t_{j+1}), \\
\ds
u(x,t_j)= u(x,t_j-)1_{\{F^{-1}[a_j,\iy)\}}(x) & 1\le j \le H-1,
\\
u(x,0)=u_0(x)1_{\{F^{-1}[a_0,\iy)\}}(x).
\end{cases}
\end{equation}
It is clear, by induction, that this set uniquely determines a solution,
denoted $u^{(m)}$. If
$\zeta^{(m)}_t(dx)=u^{(m)}(x,t)dx$, $t\in[0,T)$,
then the map $M\ni m\mapsto\zeta^{(m)}\in D([0,T),\calM_+(\R^d))$
is denoted by $\calZ$.

Step 2. {\it A continuity property of $\calZ$}. We prove the following claim.
There exists a function $\gam:\R_+\to\R_+$ such that
$\gam(0+)=0$ and the following holds.
Let $k,m\in M$ and $w=u^{(k)}$, $v=u^{(m)}$.
Denote $\Del=w-v$. Then
\begin{equation}\label{b6}
\sup_{t<T}\|\Del(\cdot,t)\|_1\le\gam(\|m-k\|^*_T),
\qquad k,m\in M.
\end{equation}

To prove the claim, denote
\[
\eta(k,m)=\max_{j\le H-1}\zeta_T(F^{-1}(a_j(k)\w a_j(m),a_j(k)\vee a_j(m))).
\]
Note that
\begin{align}\label{b8}
w(x,t)&=w_0(x)+\int_0^t\int_{\R^d}w(y,s)\rho(y,x)dyds,
\qquad (x,t):F(x)>k_t,t>0,
\\
\label{b9}
v(x,t)&=v_0(x)+\int_0^t\int_{\R^d}v(y,s)\rho(y,x)dyds,
\qquad (x,t):F(x)>m_t,t>0,
\end{align}
and $w$ (resp., $v$) vanishes to the left of $k$ (resp., $m$).
Also note that $w(x,t)\vee v(x,t)\le z(x,T)$. For $t\in[0,T)$,
\begin{align*}
\|\Del(\cdot,t)\|_1 &\le
\|\Del(\cdot,0)\|_1
+\int_{F^{-1}(k_t,\iy)\triangle F^{-1}(m_t,\iy)}(w\vee v)(x,t)dx
\\
&\qquad
+\int_{F^{-1}(k_t\vee m_t,\iy)}\int_0^t\int_{\R^d}
|\Del(y,s)|\rho(y,x)dydsdx
\\
&\le\|\Del(\cdot,0)\|_1+\eta(k,m)+\tilde c\int_0^t\|\Del(\cdot,s)\|_1ds.
\end{align*}
Since $\|\Del(\cdot,0)\|_1\le\eta(k,m)$, Gronwall's lemma gives
\[
\|\Del(\cdot,t)\|_1\le 2\eta(k,m)e^{\tilde cT},
\qquad t<T.
\]
Hence \eqref{b6} will be proved once it is shown that
$U(0+)=0$, where
\[
U(\kap)=\sup\{\zeta_T(F^{-1}(b,b+\kap)):b\in\R\}, \qquad \kap>0.
\]
Assuming the contrary, there exists $\kap'>0$ and
$\{b_n\}$ such that, along a subsequence,
\begin{equation}\label{b7}
\zeta_T(F^{-1}(b_n,b_n+n^{-1}))>\kap'.
\end{equation}
Now, $\{b_n\}$ must be bounded. For if there is a subsequence with $b_n\to\iy$,
by the continuity of $F$, for every compact $K$
the set $F^{-1}(b_n,\iy)\cap K$ must be empty for large $n$,
which shows that \eqref{b7} cannot hold.
If there is a subsequence $b_n\to-\iy$ then
$F^{-1}(-\iy,b_n+n^{-1})\cap K$ must be empty for all large $n$
and again \eqref{b7} cannot hold.
Hence $b_n$ is bounded. Let $b$ be a limit.
Then, given $\kap_1>0$, $(b_n,b_n+n^{-1})\subset(b-\kap_1,b+\kap_1)$
for large $n$ along a subsequence, showing
\[
\kap' \le \zeta_T(F^{-1}(b-\kap_1,b+\kap_1)).
\]
Hence $\zeta_T(F^{-1}\{b\})>0$, contradicting our assumption
$\leb(F^{-1}\{a\})=0$ for every $a$. This proves \eqref{b6}.

Step 3. {\it Particle systems with piecewise constant boundary}.
Given $m\in M$ we construct a particle system
in which, for every $t\in[0,T)$, all particles lie within $F^{-1}[m_t,\iy)$.
Its initial configuration is given by the restriction of $\{x_i\}$
to $F^{-1}[m_0,\iy)$. During any interval $(t_j,t_{j+1})$,
the reproduction is according to $\rho$, but every newborn outside of
$F^{-1}[a_j(m),\iy)$ is immediately removed.
If for $j\in\{1,\ldots,H-1\}$, $t_j$ is a continuity point of $m$,
nothing happens at this time. Otherwise,
$m$ necessarily performs a positive jump,
hence the domain $F^{-1}[m,\iy)$
decreases. At this time, all particles outside $F^{-1}[a_j(m),\iy)$
are removed.
Denote the configuration process by $\zeta^{(m),N}$.

We show that for every $m\in M$,
$\bar\zeta^{(m),N}\to\zeta^{(m)}:=\calZ(m)$ in probability,
uniformly on $[0,T)$.
For the initial condition, we have
$\bar\zeta^{(m),N}_0=\bar\xi^N_01_{\{F^{-1}[a_0,\iy)\}}$.
The boundary of the domain $F^{-1}[a_0,\iy)$ has Lebesgue measure
zero, as follows from Assumption \ref{assn10}(i) upon noting that
$\pl F^{-1}[a,\iy)\subset F^{-1}\{a\}$ for every $a\in\R$.
Because the limit $\xi_0$ has a density, this implies
$\bar\zeta^{(m),N}_0\to\xi_01_{\{F^{-1}[a_0,\iy)\}}$ in probability.

Assume that for $j\ge0$, $\bar\zeta^{(m),N}_{t_j}\to\zeta^{(m)}_{t_j}$.
Then the convergence on the interval $(t_j,t_{j+1})$ to a measure-valued
trajectory which has a density satisfying the integro-differential equation
in \eqref{b1} is as in the proof of \cite[Proposition 2.1]{dur11};
the proof there, based on \cite{fou-mel},
is for $d=1$ and the domain $\R$, but the same proof is valid for
$d\ge1$ an arbitrary domain of $\R^d$, because that is
the generality of the results of \cite{fou-mel}. Because the trajectory
is continuous on $(t_j,t_{j+1})$, the convergence is uniform there.
Next, at the trimming time $t_{j+1}$, the convergence follows
by the same argument as for $t=0$.

Step 4.
{\it Relation to the original particle system}.
Consider now, for each $N$, two stochastic processes,
$k^N$ and $m^N$, with sample paths in $M$, defined as follows.
For $x\in\R$, let $P^-_\eps(x)=\max\{y\in\eps\Z:y\le x\}$,
$P^+_\eps(x)=\min\{y\in\eps\Z:y\ge x\}$, and let
\[
k^N_t=
P^-_\eps\Big(\inf_{s\in[t_j,t_{j+1})}\ell^N_s\Big)-\eps,
\qquad
m^N_t=
P^+_\eps\Big(\sup_{s\in[t_j,t_{j+1})}\ell^N_s\Big)+\eps,
\qquad
t\in[t_j,t_{j+1}).
\]
By construction,
\begin{equation}\label{b3}
k^N_t<\ell^N_t<m^N_t, \qquad t\in[0,T).
\end{equation}
Constructing a particle system with piecewise constant boundary,
as in Step 2, makes perfect sense even when the trajectories are random
members of $M$ that anticipate the future of $\xi^N$, such as $k^N$ and $m^N$. To construct these systems,
one first evaluates $\ell^N$, and based on it, $k^N$ and $m^N$,
and then uses these random trajectories
to construct the particle systems as in Step 2.
A key point is that, thanks to \eqref{b3},
one can couple the three systems in such a way that the two auxiliary configuration measures form lower and upper bounds on $\xi^N$ in the sense of measure inequality. In particular, we have the following.
\begin{lemma}\label{lem0a}
There exists a coupling  such that, a.s.,
\begin{equation}\label{b5}
\zeta^{(m^N),N}_t\sqsubset\xi^N_t\sqsubset\zeta^{(k^N),N}_t,
\qquad t\in[0,T),
\end{equation}
where the two particle configurations corresponding to the piecewise constant boundaries $k^N$ and $m^N$ are denoted by $\zeta^{(k^N),N}$ and $\zeta^{(m^N),N}$.
\end{lemma}
The standard but tedious proof of this lemma appears in the appendix.


Since the maps $P^\pm_\eps$ are not continuous,
$k^N$ and $m^N$ may not converge in law. However, the tightness
of the laws of $\ell^N$ and the specific structure of $M$
clearly give tightness of $(k^N,m^N)$. Consider then any convergent subsequence
$(\bar\xi^N,\ell^N,k^N,m^N)\To(\xi,\ell,\hat k,\hat m)$.
We claim that the limit in law of
$(\bar\xi^N,\ell^N,k^N,m^N,\bar\zeta^{(k^N),N},\bar\zeta^{(m^N),N})$
exists and is given by
\[
(\xi,\ell,\hat k,\hat m,\calZ(\hat k),\calZ(\hat m)).
\]
To simplify the notation, we prove the claim only for the 3-tuple
$(\bar\xi^N,m^N,\bar\zeta^{(m^N),N})$;
the proof for the 6-tuple is very similar.
Because the claim is concerned with convergence
of probability measures, it suffices to prove vague convergence.
Let $f:D_T\to\R$ be compactly supported, where
$D_T:=D([0,T],\calM(\R^d))\times D([0,T],\R)\times D([0,T],\calM(\R^d))$.
Let $M_f$ be the set of $m\in M$ for which $f(\cdot,m,\cdot)\ne0$
and note that $\# M_f<\iy$ because
for $c>0$ there are only finitely many $m\in M$ with
$\|m\|^*_T<c$. By the uniform continuity of $f$, one has
$\gam_f(0+)=0$, where
\[
\gam_f(\kap)=\sup\{|f(\al,m,\beta)-f(\al,m,\beta')|:
(\al,m,\beta)\in D_T,\,d_{\rm L}(\beta,\beta')<\kap\}.
\]
Now,
\begin{align*}
\E f(\bar\xi^N,m^N,\bar\zeta^{(m^N),N})
&=\sum_{m\in M_f}\E[1_{\{m\}}(m^N) f(\bar\xi^N,m,\bar\zeta^{(m),N})]
\\
&=
\sum_{m\in M_f}\E[1_{\{m\}}(m^N) f(\bar\xi^N,m,\calZ(m))]
+W^N,
\end{align*}
where
\[
|W^N|\le\sum_{m\in M_f}\EE[\gam_f(d_{\rm L}(\bar\zeta^{(m),N},
\calZ(m))]\to0,
\]
by the convergence in probability proved in Step 3.
Because each $m\in M$ is isolated from all other members of $M$,
the convergence $(\bar\xi^N,m^N)\To(\xi,\hat m)$
implies $\E[1_{\{m\}}(m^N)g(\bar\xi^N)]
\to\E[1_{\{m\}}(\hat m)g(\xi)]$ whenever $g$ is bounded and continuous.
This gives
\[
\E f(\bar\xi^N,m^N,\bar\zeta^{(m^N),N})
\to\sum_{m\in M_f}\E[1_{\{m\}}(\hat m)f(\xi,m,\calZ(m))]
=\E[f(\xi,\hat m,\calZ(\hat m))],
\]
proving the claim.

As a consequence of the above convergence and \eqref{b5},
we obtain a bound on the density $u$, namely,
for every $t$ and a.e.\ $x$,
\begin{equation}\label{b10}
\hat v(x,t) \le u(x,t) \le \hat w(x,t),
\end{equation}
where $\hat v=u^{(\hat m)}$ and $\hat w=u^{(\hat k)}$
are the random densities corresponding to
$\calZ(\hat m)$ and $\calZ(\hat k)$.

Step 5.
{\it Completing the proof}.
Denoting $\tilde\Del=u-\hat v$,
$\hat\Del=\hat w-\hat v$, and $\hat q=\|\hat m-\hat k\|^*_T$,
we have by \eqref{b10}, $\|\tilde\Del(\cdot,t)\|_1
\le\|\hat\Del(\cdot,t)\|_1\le \gam(\hat q)$,
where \eqref{b6} is used.

Note that $\hat w$ and $\hat v$ satisfy \eqref{b8}--\eqref{b9},
with $k$ and $m$ replaced by the random
$\hat k$ and $\hat m$.
Let $(x,t)$ be such that $F(x)>\hat m_t$, by which $F(x)>\ell_t\ge\hat k_t$.
Then
\begin{align*}
u(x,t)&=\hat v(x,t)+\tilde\Del(x,t)\\
&=u_0(x)-\tilde\Del_0(x)+\int_0^t\int_{\R^d}u(y,s)\rho(y,x)dyds
-\int_0^t\int_{\R^d}\tilde\Del(y,s)\rho(y,x)dyds
+\tilde\Del(x,t).
\end{align*}
Denoting, for $(x,t)\in\R^d\times[0,T)$,
\[
\Ps(x,t)=\Big|u(x,t)-u_0(x)-\int_0^t\int_{\R^d}u(y,s)\rho(y,x)dyds\Big|,
\]
we therefore have
\[
\int_{\R^d}\Ps(x,t)1_{\{F(x)>\hat m_t\}}dx
\le (2+\tilde cT)\gam(\hat q).
\]
Moreover,
\[
\Ps(x,t)\le 2z(x,T)+T\int_{\R^d}z(y,T)\tilde c \tilde\rho(x-y)dyds.
\]
Analogously to $U$, define
$\tilde U(\kap)=\sup\{\int_{F^{-1}(b,b+\kap)}\tilde\rho(x)dx: b\in\R\}$,
and note, by similar reasoning, that $\tilde U(0+)=0$.
Then
\[
\int_{\R^d}\Ps(x,t)1_{\{\ell_t<F(x)\le\hat m_t\}}dx
\le 2U(\hat q)+T\tilde c \,|\zeta_T|\, \tilde U(\hat q).
\]
Hence there exists a constant $c$ such that
\[
\psi:=\sup_{t<T}\int_{\R^d}\Ps(x,t)1_{\{F(x)>\ell_t\}}dx
\le c(\gam(\hat q)+U(\hat q)+\tilde U(\hat q)).
\]
Given $\eps,\del'>0$ choose $\del>0$ so small that
$\PP(\Om_{\del,\eps})>1-\del'$, where $\Om_{\del,\eps}=\{w_T(\ell,\del)<\eps\}$.
By the definition of $m^N$ one has for every $j$,
$m^N_{t_j}\le \sup_{[t_j,t_{j+1})}\ell^N+2\eps$, hence
$\hat m_{t_j}\le\sup_{[t_j,t_{j+1})}\ell+2\eps$. An analogous statement
holds for $\hat k$. This gives, on $\Om_{\del,\eps}$,
$\hat m_t-\hat k_t\le5\eps$ for all $t<T$, hence $\hat q\le 5\eps$.
As a result,
\[
\PP(\psi > c(\gam(5\eps)+U(5\eps)+\tilde U(5\eps)))<\del'.
\]
Finally, $\del'\to0$ followed by $\eps\to0$ shows $\psi=0$ a.s.
This shows \eqref{b2} and completes the proof.
\end{proof}

\appendix

\section{Proof of Lemma \ref{lem0a}}

The three particle systems will be constructed out of a single branching random walk (with labeled particles), which is a non-local branching model without selection. In this system, the initial $N$ particles are located at $x^i$, $1\le i\le N$. Each particle has a rate-$1$ Poisson clock according to which it gives birth at location drawn according to $\rho(y,\cdot)$, $y$ being the particle's location. The first $N$ particles are labeled $1,\ldots,N$, and for $i\ge N+1$, the $i$-th particle to be born (regarding the initial $N$ particles as being born at $0$) is labeled $i$. Let the birth time of particle $i\ge1$ be denoted by $\sig_i$. Let $\tau_{i,j}$, $j\ge1$, denote the times particle $i$ gives birth, and let $\tau_i(t)=\min\{\tau_{i,j}:\tau_{i,j}>t\}$ denote the next time after $t$ when particle $i$ gives birth.
Let $X_i$ denote the position of particle $i$. With this notation, the configuration measure of the branching random walk is
\[
\zeta^N_t=\sum_{i\in\calQ(t)}\bdel_{X_i},
\qquad
\calQ(t)=\{i:t\ge\sig_i\}.
\]

For a nonempty finite set $S\subset\N$, denote by $\argmin S$ the particle $i\in S$ for which one has $F(X_i)\le F(X_j)$ for all $j\in S$, with ties broken according to the particle labels.

The system studied in this paper can be represented as
\[
\xi^N_t=\sum_{i\in\calR(t)}\bdel_{X_i},
\]
where $\calR(t)\subset\calQ(t)$ is defined inductively as follows.
Starting with time $0$, let $\calR(0)=\{1,\ldots,N\}$ and $\theta_0=0$. Next, for $k\ge0$, given $\theta_k$ and $\{\calR(t), t\le \theta_k\}$, define $\theta_{k+1}$ and $\{\calR(t), t\le \theta_{k+1}\}$ in the following way.

Define $\theta_{k+1}=\min\{\tau_i(\theta_k):i\in\calR(\theta_k)\}$ and let $i_{k+1}$ be the particle born at $\theta_{k+1}$. Thus, up to $\theta_{k+1}$ no particle among $\calR(\theta_k)$ gives birth, whereas at $\theta_{k+1}$, a particle in this set gives birth to particle $i_{k+1}$. Let $\calR(t)=\calR(\theta_k)$ for $t\in(\theta_k,\theta_{k+1})$ and
\[
\calR(\theta_{k+1})=\calR(\theta_k)\cup\{i_{k+1}\}- \argmin(\calR(\theta_k)\cup\{i_{k+1}\}).
\]
It is clear that $\theta_k\to\iy$ a.s.\ as $k\to\iy$, and therefore $\calR(t)$ is defined for all $t$, a.s., completing the construction of $\xi^N_t$.

Recall that the set $M$ consists of piecewise constant trajectories whose jump times are denoted by $t_j$. Let $T(t)=\min\{t_j:t_j>t\}$. For an arbitrary $m\in M$ (where $m$ will later stand for either $m^N$ or $k^N$), let
\[
\zeta^{(m),N}_t=\sum_{i\in\calS(m;t)}\bdel_{X_i}.
\]
Here, $\calS(m;t)$, abbreviated to $\calS(t)$, is constructed as follows. Initially, $\calS(0)=\{1,\ldots,N\}$, and $s_k=0$. Given $s_k$ and the construction of $\calS(t)$, $t\le s_k$, let
\[
s_{k+1}=\min\{\tau_i(s_k):i\in\calS(s_k)\}\w T(s_k).
\]
Thus at $s_{k+1}$ either (a) one of the particles in $\calS(s_k)$ gives birth, or (b) one of the $t_j$ has been reached. In case (a), denote the newborn particle by $i_{k+1}$ and let
\[
\calS(s_{k+1})=\begin{cases}\calS(s_k)\cup\{i_{k+1}\}
& \text{if } F(X_{i_{k+1}})\ge m(s_k),
\\
\calS(s_k) & \text{if } F(X_{i_{k+1}})< m(s_k).
\end{cases}
\]
Thus the new particle is accepted if its fitness is $\ge m(s_k)=m(s_{k+1})$. In case (b),
\[
\calS(s_{k+1})=\calS(s_k)-\{i\in\N:F(X_i)<m(s_{k+1})\},
\]
rejecting all particles whose fitness is $<m(s_{k+1})$. This completes the construction of $\zeta^{(m),N}$.

In view of the inequalities \eqref{b3}, the lemma will be proved once it is shown that
\begin{equation}\label{e-2}
\text{
if $m(t)<\ell^N(t)$ for all $t$ then $\calR(t)\subset\calS(t)$ for all $t$,
}
\end{equation}
\begin{equation}\label{e-1}
\text{
if $m(t)>\ell^N(t)$ for all $t$ then $\calS(t)\subset\calR(t)$ for all $t$.
}
\end{equation}
We only provide the proof of \eqref{e-2} as that of \eqref{e-1} is similar.
Arguing by contradiction, assume \eqref{e-2} is false, and let $s$ be the first time it is violated. That is, at $s$, for the first time, there is a particle in $\calR(s)$ that is not in $\calS(s)$. At such $s$, one of these must occur:

- A new particle whose fitness is $\ge\ell^N(s)$ has been introduced into $\calR(s)$ (and not immediately rejected), that is a child of a particle in the set $\calR(s-)$. Because $\calR(s-)\subset\calS(s-)$, this particle was also introduced into $\calS(s)$, and because $m(s)<\ell^N(s)$, it was not immediately rejected. This contradicts the assumption $\calR(s)\not\subset\calS(s)$.

- At $s$, $m$ makes a jump upward, that is, $m(s)>m(s-)$. As a result, some particles in $\calS(s-)$ are removed and do not appear in $\calS(s)$, whereas at least one of them appears in $\calR(s)$. Thus the fitness of this particle is $\ge\ell^N(s)$, which contradicts the fact that it also must be $<m(s)$.

This completes the proof of Lemma \ref{lem0a}.
\qed

\vspace{1em}

\noi
{\bf Acknowledgment.}
I thank two anonymous referees for their thoughtful and helpful suggestions.
This research was supported by ISF grant 1035/20.

\bibliographystyle{is-abbrv}

\bibliography{revision1}

\end{document}